\documentclass[11pt]{article}

\usepackage{amsmath,amsthm,amsfonts,amssymb,mathrsfs,epsfig,bm}

\def\ch{\hbox {\rm cosh\,}}

\def\sh{\hbox {\rm sinh\,}}

\newcommand{\rr}{\stackrel {d}{=}}

\renewcommand{\Re}{{\rm I\kern-0.16em R}}

\def\@begintheorem#1#2{\trivlist \item[\hskip \labelsep{\bf #1\ #2}]}
\def\@opargbegintheorem#1#2#3{\trivlist
      \item[\hskip \labelsep{\bf #1\ #2\ (#3)}]}

\newtheorem{proposition}{Proposition}

\newtheorem{thm}[proposition]{Theorem}
\newtheorem{lemma}[proposition]{Lemma}
\newtheorem{corollary}[proposition]{Corollary}
\newtheorem{example}[proposition]{Example}
\newtheorem{remark}[proposition]{Remark}

\def\ch{\hbox {\rm cosh}}
\def\sh{\hbox {\rm sinh}}
\def\e{\hbox {\rm e}}
\def\P{{\bf P}}
\def\R{{\bf R}}
\def\R{{\bf R}}

\def\E{{\bf E}}

\def\cC{{\cal C}}

\def\rp{\right)}
\def\lp{\left(}

\def\al{\alpha}

\def\la{\lambda}
\def\om{\omega}

\begin{document}

\author{
Paavo Salminen\\{\small \AA bo Akademi University,}
	\\{\small Mathematical Department,}
	\\{\small F\"anriksgatan 3 B,}
	\\{\small FIN-20500 \AA bo, Finland,} 
	\\{\small \tt phsalmin@abo.fi}
\and
Marc Yor\\
	{\small Universit\'e Pierre et Marie Curie,}\\
	{\small Laboratoire de Probabilit\'es }\\
	{\small et Mod\`eles al\'eatoires,}\\
	{\small 4, Place Jussieu, Case 188}\\
	{\small F-75252 Paris Cedex 05, France}
}
\vskip5cm

\title{
	On hitting times of affine boundaries by reflecting Brownian motion and Bessel
        processes}

\maketitle

\begin{abstract}
  Firstly, we compute the distribution function for the hitting time of a linear
  time-dependent boundary $t\mapsto a+bt,\ a\geq 0,\,b\in \R,$ by a reflecting Brownian motion. 
  The main tool hereby is Doob's formula which gives the probability that Brownian motion started inside a wedge  does not hit this wedge. Other key ingredients are the time inversion property of Brownian motion and the time reversal property of diffusion bridges. Secondly, this methodology can also be applied for the three dimensional Bessel process. Thirdly, we consider Bessel bridges from 0 to 0 with dimension parameter $\delta>0$ and show that the probability that such a Bessel bridge crosses an affine boundary is equal to the probability that this  Bessel bridge stays below some fixed value.
\end{abstract}

\section{Introduction}
\label{sec1}

Let $B=\{B_t: t\geq 0\}$ denote a real valued Brownian motion and $f:[0,+\infty)\mapsto \R$ a "nice" function. Introduce the first hitting time of $f$ by $B$ via
\begin{equation}
\nonumber	
H_f:= \inf\{t:\, B_t=f(t)\}.
\end{equation}
There are many important applications (e.g. in sequential analysis, see, e.g., Lerche \cite{lerche86}, and mathematical finance, especially in the theory of American options, see, e.g., Peskir and Shiryaev \cite{peskirshiryaev06} Chapter VII) which call for the distribution of $H_f$ for some particular functions $f.$ For linear boundaries 
$f(t)=a+bt$ we refer to Doob \cite{doob49} and Durbin \cite{durbin71}, for square root boundaries $f(t)=\sqrt{a+bt}$ see Breiman \cite{breiman67}, Ciffarelli et al. \cite{cdr73}, Shepp \cite{shepp82}, Delong \cite{delong81},\cite{delong83} and Yor \cite{yor84}, and for parabolic boundaries $f(t)=a+bt^2,$ Groeneboom \cite{groeneboom89}, Salminen \cite{salminen88} and Martin-L\"of \cite{martinlof98}. We refer to Durbin \cite{durbin92} (with appendix by Williams) for a practical method via integral equations to calculate the hitting probabilities of time dependent boundaries. For the method of images, see Lerche \cite{lerche86}. In  Alili and Patie \cite{alilipatie05} a new approach based on functional transformations of processes is developed and its connection with the method of images is studied.

In this paper we compute the distribution of the hitting time of a linear boundary for
\begin{description}
\item[$\bullet$] reflecting Brownian motion (RBM),
\item[$\bullet$] three dimensional Bessel process (BES).
\end{description} 

Our approach is based on a formula by Doob presented in \cite{doob49}, as well as the time inversion property for Brownian motion and Bessel processes, and the time reversal property of diffusion bridges. In fact, the case of RBM has been already studied via Doob's formula and the time inversion by Abundo (\cite{abundo02}; Corollary 3.4) for straight lines with positive slope. For Bessel processes with dimension parameter $\delta>0$ results for hitting times of straight lines through origin can be found in Pitman and Yor 
(\cite{pitmanyor81};  Section 8 p. 332) and in Alili and Patie (\cite{alilipatie10}; Theorem 5.1)  - we review these results below in Remarks \ref{piyo} and Theorem \ref{alpa}, respectively. 

The above description shows that one of the aims of our paper is to relate formulae about these (classical) first hitting times which appear to be rather scattered in the literature. Moreover, seemingly different formulae for the same densities  may be found in different papers but this discrepancy can often be explained, e.g., via the Poisson summation formula.

We now present our main notation. All processes we study are
defined on the canonical space $C$ of continuous functions $\omega:\R_+\mapsto \R.$ 
Let 
$$
{\cal C}_t:=\sigma\{\omega_s:=\omega(s): s\leq t\}
$$  
denote the smallest $\sigma$-algebra making the coordinate mappings up to time $t$ measurable 
and take ${\cal C}$ to be the smallest $\sigma$-algebra including all $\sigma$-algebras ${\cal C}_t,\ t\geq 0.$ Processes are identified in $(C,{\cal C})$ via their probability measures; e.g., 
$\P^{BM}_x,$ $\P^{RBM}_x,$ and $\P^{BES(\delta)}_x$ denote the probabilities under which 
the coordinate process with $\omega_0=x$ is a Brownian motion, a reflecting Brownian motion and a  Bessel process with dimension parameter $\delta$ , respectively. 

Throughout the paper the notation
\[ 
H_{a,b}:=\inf\{t\,:\, \omega_t=a+bt\}
 \]
is used for the hitting time of a linear time dependent boundary. 
Notice that for $x=0$, by the scaling property of Brownian motion and Bessel processes,  
\begin{equation} 
\label{scale1}
 H_{a,b} \rr 
\begin{cases}
b^{-2}H_{ab,1}, & b>0,\\
b^{-2}H_{a|b|,-1},& b<0,
\end{cases}
\end{equation}
and if $f_{a,b}$ denotes the density of  $H_{a,b}$ then
\begin{equation} 
\label{scale2}
 f_{a,b}(u) =
\begin{cases}
b^{-2}f_{ab,1}(b^2u), & b>0,\\
b^{-2}f_{a|b|,-1}(b^2u),& b<0.
\end{cases}
\end{equation}

Also for $x=0,$ the time inversion property - a basic tool in our approach - is valid for Brownian motion and Bessel processes with dimension parameter $\delta>0$ initiated at 0 (for $0<\delta<2$ the Bessel process is assumed to be reflecting at 0), that is, under $\P^{BM}_0$ and $\P^{BES(\delta)}_0$
\begin{equation} 
\label{timeinv}
\{\omega_t: t>0\}\,\rr\,\{t\,\omega_{1/t}:t>0\}.
\end{equation}  
We remark that this property can also be used to derive
the distribution of the last hitting time from the distribution of the first hitting time (and vice versa). Indeed, for $a,b>0$ let 
$$ 
G_{b,a}:=\sup\{t\,:\, \omega_t=b+at\}.
$$ 
Then, under $\P^{BM}_0$ and $\P^{BES(\delta)}_0,$ it holds 
\begin{eqnarray*}
&&
\hskip-1cm
\sup\{t\,:\, \omega_t=b+at\}\,\rr\,\sup\{t>0\,:\, t\omega_{1/t}=b+at\}
\\
&&
\hskip2.6cm
=\sup\{t\,:\, \omega_{1/t}=\frac bt+a\}
\\
&&
\hskip2.6cm
=1/\inf\{u\,:\, \omega_{u}=a+bu\},
\end{eqnarray*}
and, hence, 
\begin{equation} 
\label{timeinv2}
G_{b,a}\,\rr\,1/H_{a,b},
\end{equation}
in other words, e.g., for $a,b>0$
\begin{equation} 
\label{timeinv3} 
\P^{BES(\delta)}_0\left(G_{b,a}<t\right)=\P^{BES(\delta)}_0\left(H_{a,b}>\frac 1 t\right).
\end{equation}

Next we recall well known formulas for the hitting times of constant boundaries by RBM and BES. Firstly,
\begin{equation} 
\label{a0}
\E^{RBM}_0\left(\exp(-\lambda H_{a,0})\right)=\left(\ch(a\sqrt{2\lambda})\right)^{-1}
\end{equation} 
from which using series expansion and  inverting term by term  we obtain 
\begin{equation} 
\label{a01}
\P^{RBM}_0\left(H_{a,0}\in du\right)/du= \frac{-2a}{\sqrt{2\pi
    u^3}}\sum_{k=-\infty}^{\infty}(4k-1)\,{\rm
  e}^{-(4k-1)^2 a^2/2u}.
\end{equation}
Making an elementary parity manipulation one can find a different
  but equivalent form for the density of $H_{a,0}$
\begin{equation} 
\label{a00}
\P^{RBM}_0\left(H_{a,0}\in du\right)/du= \frac{a}{\sqrt{2\pi
    u^3}}\sum_{k=-\infty}^{\infty}(-1)^k(2k+1)\,{\rm
  e}^{-(2k+1)^2 a^2/2u}
\end{equation}
displayed in \cite{borodinsalminen02} p. 355 and 641.

For three dimensional Bessel process it holds
\begin{equation} 
\label{a10}
\E^{BES(3)}_0\left(\exp(-\lambda H_{a,0})\right)=\frac{a\sqrt{2\lambda}}{\sh(a\sqrt{2\lambda})},
\end{equation} 
and inverting 
\begin{eqnarray} 
\label{a11}
&&
\P^{BES(3)}_0\lp H_{a,0}\in du\rp /du
\\
&&
\nonumber
\hskip2cm
= \frac a{\sqrt{2\pi u^5}}\,
\sum_{k=-\infty}^\infty\lp (2k+1)^2a^2-u\rp  {\rm  e}^{-(2k+1)^2a^2/2u},
\end{eqnarray} 
which coincides with the formula presented in \cite{borodinsalminen02} p. 463
letting therein $x\downarrow 0.$ We refer also to Pitman and Yor \cite{pitmanyor03} where the subordinators, denoted by $C_t$ and $S_t,$  generated by the positive  powers of the right hand sides of (\ref{a0}) and (\ref{a10}), respectively, are studied. 

For ease of the reader (and of ourselves !), when comparing certain theta function type series indexed by {\bf Z} or {\bf N}, we indulge in stating the following elementary lemma.

\begin{lemma}
Let $\psi$ be a function defined on odd integers $2k+1,\, k\in{\bf Z}.$ Assume that there exists $\varepsilon$ such that $\psi(2k+1)=\varepsilon\,\psi(-2k-1)$ for all $k\in{\bf Z}.$  Then
\begin{equation}
\label{odd}
\nonumber
\sum_{n=-\infty}^{\infty}\psi(2n+1)=(1+\varepsilon)\sum_{k=0}^{\infty}\psi(2k+1),
\end{equation}
where it is assumed that the sum on the RHS is finite. In particular, if $\psi$ is even, i.e., $\varepsilon=+1$ then
\begin{equation}
\label{even}
\nonumber
\sum_{n=-\infty}^{\infty}\psi(2n+1)=2\sum_{k=0}^{\infty}\psi(2k+1),
\end{equation}
and if $\psi$ is odd, i.e., $\varepsilon=-1$ then
\begin{equation}
\label{odd2}
\nonumber
\sum_{n=-\infty}^{\infty}\psi(2n+1)=0.
\end{equation}
\end{lemma}

The outline of the paper is as follows:  Doob's formula is discussed in the next section, and in the third section the distribution of  $H_{a,b}$ for RBM is presented. In the fourth section the BES case is treated, and the paper concludes with a fifth section, where the distribution of the maximum of a general Bessel bridge is connected with the distribution of $H_{a,b}$ for this Bessel bridge.  

\section{Doob's formula}
\label{sec2}

In his famous paper on the Kolmogorov-Smirnov test, Doob \cite{doob49} derived the expression for the probability that Brownian motion started inside a space-time wedge does not hit this wedge. We first give some symmetry properties of the function characterizing this probability.

\begin{proposition}
\label{sym0}
For $\al,\beta,a,b\geq 0$ let
\begin{eqnarray}
\label{de1}
&&
\hskip-1cm
G(\al,\beta;a,b):=\P^{BM}_0\left(-(\al +\beta t)\leq \omega_t\leq a+b t \quad \forall\ t\geq
0\right).
\end{eqnarray}
Then  $G$ has the symmetry properties
\begin{equation}
\label{G1}
G(\al,\beta;a,b)=G(\beta,\al;b,a)=G(a,b;\al,\beta),
\end{equation}
and for all $c>0$ the scaling property 
\begin{equation}
\label{G2}
G(\al,\beta;a,b)=G(\frac\al c,\beta c;\frac ac,b c).
\end{equation}
\end{proposition}

\begin{proof}  
The first equality in (\ref{G1}) follows from  time inversion:
\begin{eqnarray*}
&&
\hskip-1cm 
G(\al,\beta;a,b)=\P^{BM}_0\left(-(\al +\beta t)\leq \omega_t\leq a+b t \quad \forall\ t>0
\right)
\\
&&
\hskip1.1cm
= 
\P^{BM}_0\left(-(\al +\beta t)\leq t\omega_{1/t}\leq a+b t \quad \forall\ t>0
\right)
\\
&&
\hskip1.1cm
= 
\P^{BM}_0\left(-(\frac\al t +\beta )\leq \omega_{1/t}\leq \frac at+b \quad \forall\ t>0
\right)
\\
&&
\hskip1.1cm
=G(\beta,\al;b,a).
\end{eqnarray*}
The second equality in (\ref{G1}) is obtained by the spatial symmetry of Brownian motion. For  (\ref{G2}) consider
\begin{eqnarray*}
&&
\hskip-1cm 
G(\al,\beta;a,b)=\P^{BM}_0\left(-(\al +\beta t)\leq \omega_t\leq a+b t \quad \forall\ t>0
\right)
\\
&&
\hskip1.1cm
= \P^{BM}_0\left(-(\al +\beta c^2 t)\leq \omega_{c^2t}\leq a+b c^2t \quad \forall\ t>0
\right)
\\
&&
\hskip1.1cm
= \P^{BM}_0\left(-(\frac \al c +\beta c t)\leq \frac 1 c\,\omega_{c^2t}\leq \frac ac+b ct \quad \forall\ t>0
\right)
\\
&&
\hskip1.1cm
=G(\frac \alpha c,\beta c;\frac ac,bc),
\end{eqnarray*}
where Brownian scaling has been used.
\end{proof}

We now state Doob's formula; for the proof, see \cite{doob49}.
\begin{thm}
For $\al,\beta,a,b\geq 0$ 
\begin{equation}
\label{G}
G(\al,\beta;a,b)=
1-\sum_{k=1}^\infty\left({\rm e}^{-2A_k}+{\rm e}^{-2B_k}-{\rm
  e}^{-2C_k}-{\rm e}^{-2D_k}\right),
\end{equation}
where
$$
A_k:=A_k(\al,\beta;a,b):=k^2ba+(k-1)^2\beta\al+k(k-1)(b\al+a\beta),
$$
$$
B_k:=B_k(\al,\beta;a,b):=(k-1)^2ba+k^2\beta\al+k(k-1)(b\al+a\beta),
$$
$$
C_k:=C_k(\al,\beta;a,b):=k^2(ba+\beta\al)+ k(k-1)b\al+k(k+1)a\beta,
$$
$$
D_k:=D_k(\al,\beta;a,b):=k^2(ba+\beta\al)+ k(k+1)b\al+k(k-1)a\beta.
$$
\end{thm}

\begin{remark} The symmetry properties of $G$ can, of course, be seen also from the expressions for $A_k, B_k, C_k,$ and $D_k.$ Indeed,
$$
A_k(\al,\beta;a,b)=B_k(a,b;\al,\beta),
$$
and, since,
$$
C_k(\al,\beta;a,b)=k^2(ba+\beta\al)+ k^2(b\al+a\beta)+k(a\beta-b\al)
$$
and
$$
D_k(\al,\beta;a,b)=k^2(ba+\beta\al)+ k^2(b\al+a\beta)-k(a\beta-b\al),
$$
it is easily seen also that
$$
C_k(\al,\beta;a,b)=D_k(a,b;\al,\beta).
$$
Moreover,
$$
A_k(\al,\beta;a,b)=A_k(\beta,\al;b,a),\ B_k(\al,\beta;a,b)=B_k(\beta,\al;b,a)
$$
and
$$
C_k(\al,\beta;a,b)=C_k(\beta,\al;b,a), D_k(\al,\beta;a,b)= D_k(\beta,\al;b,a).
$$
We also note that each of $A_k,B_k,C_k$ and $D_k$ satisfies the same scaling property as $G.$
Notice also that if we introduce
$$
\Upsilon(m,n):=m^2ba+n^2\beta\al+mn(b\al+a\beta)
$$
then
$$
A_k=\Upsilon(k,k-1)\quad {\sl and}\quad B_k=\Upsilon(k-1,k).
$$
\end{remark}

The limiting distribution function (found by Kolmogorov) of the Kolmogorov-Smirnov statistics  is 
$$
F_K(\lambda):= 1+2\sum_{k=1}^\infty(-1)^k\,{\rm e}^{-2k^2 \lambda^2},\ \lambda>0,
$$
and was identified by Doob to be the distribution function of the maximum of the absolute value of a standard Brownian bridge. Moreover, Doob observed that this distribution function is closely related to the probability that RBM never hits a straight line with positive slope. Indeed, taking in (\ref{de1})  $\al=a$ and  $\beta=b$ gives 
\[ 
A_k=B_k=(2k-1)^2ba\quad {\rm and}\quad
 C_k=D_k=(2k)^2ba,
  \]  
and yields this probability as displayed in the following

\begin{corollary} For $a,b\geq 0$ it holds 
\label{doob1}
\begin{eqnarray}
\label{de2}
&&
\hskip-1.5cm
\nonumber
\P^{BM}_0\left(|\omega_t|\leq a+bt\quad \forall\ t\geq
0\right)
=
\P^{RBM}_0\left(H_{a,b}=\infty\right)
\\
&&
\hskip3.5cm
= 1+2\sum_{k=1}^\infty(-1)^k\,{\rm e}^{-2k^2 ab}
\\
&&
\hskip3.5cm
\nonumber
=\Theta^*(2ab/\pi),
\end{eqnarray}
where 
$$
\Theta^*(u):=\sum_{k\in{\bf Z}}(-1)^k\,{\rm e}^{-\pi k^2 u}
$$
is the "modified" Jacobi theta function.
\end{corollary}

\begin{remark} 
\label{r00}
Trivially, the LHS of (\ref{de2}) equals
$$
\P^{BM}_0\left(\sup_{t\geq 0}(|\omega_t|-bt)\leq a\right)
$$
so that $a\mapsto\Theta^*(2a^2/\pi)$ appears as a distribution function and equals $F_K.$ 
For a more general discussion for such suprema involving Bessel processes, see  
Theorem \ref{t31}.
We note that the scaling property of Brownian motion also yields that the LHS of (\ref{de2}) is a function of the product $ab$ only.
\end{remark}

Using time inversion Abundo \cite{abundo02} (see also Scheike
\cite{scheike92}) computed the probability that a Brownian bridge stays inside 
a space-time wedge. 
Let $x,y\in\R$ and $u>0$ be given, and denote by $\P^{BM}_{x,u,y}$ the probability measure associated 
with Brownian bridge from $x$ to $y$ of length $u.$ The next result is given in \cite{abundo02}, Proposition 3.5 and Theorem 3.3. To make the exposition more self contained we give here a 
proof which differs slightly from the proof in \cite{abundo02}.

\begin{thm}
\label{abundo2}
For $u>0$ and $\al,\beta,a,b\geq 0$ 
\begin{eqnarray}
\label{ae2}
&&
\nonumber
\hskip-1cm 
\P^{BM}_{0,u,y}\left(-(\al + \beta t)\leq \omega_t\leq a +b t\quad \forall\ t\in(0,u)\right)
\\
&&
\hskip2cm 
= G\lp\al,\beta+\frac{\al+y}u;a,b+\frac{a-y}u\rp, 
\end{eqnarray}
where  $-(\al + \beta u)\leq y\leq a + b u $ and $G$ is as defined in (\ref{G}).
In particular, for $\al=a,$ $\beta=b,$ and $|y|\leq a+bu$
\begin{eqnarray}
\label{ae1}
&&
\nonumber
\hskip-1cm 
\P^{BM}_{0,u,y}\left(|\omega_t|\leq a+b t \quad \forall\ t\in(0,u)\right)
\\
&&
\hskip2cm 
= \sum_{n=-\infty}^\infty(-1)^n\,\exp\left(-2n^2 a(b+\frac a
u)+2na\frac{ y}u\right).
\end{eqnarray}
\end{thm}
\begin{proof} Using 
$$
\left(\{\omega_t;0\leq t\leq u\}, \P^{BM}_{0,u,y}\right) 
\rr
\left(\{\omega_t+\frac y u t;0\leq t\leq u\}, \P^{BM}_{0,u,0}\right) 
$$
we write
\begin{eqnarray*}
&&
\hskip-.5cm
\P^{BM}_{0,u,y}\left(-\al-\beta t\leq\omega_t\leq a+b t \quad \forall\ t\in(0,u)\right)
\\
&&
\hskip1cm 
= \P^{BM}_{0,u,0}\left(-\al-\beta t\leq\omega_t+\frac y u t\leq a+b t\quad \forall\ t\in(0,u)\right) 
\\
&&
\hskip1cm 
= \P^{BM}_{0,u,0}\left(-\al-(\beta +\frac y u) t\leq\omega_t\leq a+(b-\frac y u) t\quad \forall\ t\in(0,u)\right).
\end{eqnarray*}
To proceed recall

$$
\left(\{\omega_t;0\leq t< u\}, \P^{BM}_{0,u,0}\right) 
\rr
\left(\left\{\frac{u-t}u\,\omega_{ut/(u-t)};0\leq t< u\right\},\P^{BM}_{0}\right)
$$
and, hence,
\begin{eqnarray*}
&&
\hskip-.5cm
 \P^{BM}_{0,u,0}\left(-\al-(\beta +\frac y u) t\leq\omega_t\leq a+(b-\frac y u) t\quad \forall\ t\in(0,u)\right)
\\
&&
\hskip1cm 
= \P^{BM}_{0}\left(-\al-(\beta +\frac y u) t\leq\frac{u-t}u\,\omega_{ut/(u-t)}\leq a+(b-\frac y u) t\quad \forall\ t\in(0,u)\right)
\\
&&
\hskip1cm 
= \P^{BM}_{0}\left(-\al-(\beta +\frac y u) \frac{vu}{v+u}\leq\frac u{v+u}\,\omega_{v}\leq a+(b-\frac y u)\frac{vu}{v+u} \quad \forall\ v\geq 0\right) 
\\
&&
\hskip1cm 
= \P^{BM}_{0}\left(-\al-(\al+\beta u + y)\frac v u \leq\omega_{v}\leq a+(a+b u-y)\frac v u \quad \forall\ v\geq 0\right)
\\
&&
\hskip1cm 
=G\left(\al,\beta+\frac{\al + y}u;a,b+\frac{a- y}u\right).
\end{eqnarray*}
by the symmetry property of $G.$ This proves formula (\ref{ae2}) and (\ref{ae1}) follows immediately. 
\end{proof}

From (\ref{ae1}), integrating over $y$ leads to the distribution function of $H_{a,b},\, a,b\geq 0,$ for RBM, see Abundo \cite{abundo02} Corollary 3.4. In the next section, see Theorem \ref{t3}, we find this distribution also for negative values of $b.$

\begin{remark}
The fact that variables $a,b,y$ and $u$ appear on the RHS of (\ref{ae1}) in such a particular way can be explained again (cf. the last sentence in Remark \ref{r00}) via the scaling property.  Indeed, from the proof above (consider the case $u=1$)
\begin{eqnarray*}
&&
\hskip-.5cm
\P^{BM}_{0,1,y}\left(|\omega_t|\leq a+b t \quad \forall\ t\in(0,1)\right)
\\
&&
\hskip1cm 
= \P^{BM}_{0}\left(-a-(a+b  + y) v \leq\omega_{v}\leq a+(a+b -y) v  \quad \forall\ v\geq 0\right)
\\
&&
\hskip1cm 
= \P^{BM}_{0}\left(-a-(a+b  + y) v \leq c\,\omega_{v/c^2}\leq a+(a+b -y) v  \quad \forall\ v\geq 0\right)
\\
&&
\hskip1cm 
= \P^{BM}_{0}\left(-\frac ac-(a+b  + y)c s \leq\omega_{s}\leq \frac ac+(a+b -y)c s  \quad \forall\ s\geq 0\right).
\end{eqnarray*}
Choosing $c=a$ yields
\begin{eqnarray*}
&&
\hskip-1cm
\P^{BM}_{0,1,y}\left(|\omega_t|\leq a+b t \quad \forall\ t\in(0,1)\right)
\\
&&
\hskip1cm 
= \P^{BM}_{0}\left(|\omega_{s}+ays|\leq 1+a(a+b)s  \quad \forall\ s\geq 0\right)
\end{eqnarray*}
showing that the LHS of (\ref{ae1}) can be viewed as a function of $b':=a(a+b)$ and $\mu:=ay$ only. \end{remark}

\section{Distribution of $H_{a,b}$ for a reflecting Brownian motion}
\label{sec3}

In the next theorem we give an explicit form of the distribution
function of $H_{a,b},\, a>0, b\in\R,$ for RBM initiated at 0. For $b>0$ formula (\ref{m02}) can be found in Abundo \cite{abundo02} Corollary 3.4 (however, there seems to be a misprint in formula (3.6) in \cite{abundo02}; 
in the exponential  $n$ should be replaced by $n^2$). Clearly, if RBM is initiated from above the line then the distribution of   $H_{a,b}$ can be obtained from the distribution of $H_{a,b}$ for BM.

\begin{thm}
\label{t3} 
For $a>0$  and $b\in\R$ and $u>0$ such that $u<-(a/b)$ for $b<0$ and there is no constraint if $b>0$ it holds 
\begin{eqnarray}
\label{m02}
&&
\nonumber
\hskip-1cm
\P^{RBM}_0\left(H_{a,b}>u\right)
\\
&&
\hskip-.9cm
=\sum_{k=-\infty}^\infty(-1)^k
\e^{-2k^2ab}\,
\lp\Phi\left(\frac{a+bu-2ka}{\sqrt u}\right)-\Phi\left(-\frac{a+bu-2ka}{\sqrt u}\right)\rp,
\end{eqnarray}
where $\Phi$ denotes the standard normal distribution function, i.e.,
\begin{equation}
\label{normal}
\Phi(x):=\int_{-\infty}^x\frac 1{\sqrt{2\pi}}\,\e^{-y^2/2}dy.
\end{equation}
The density, with the same constraints, is given by 
\begin{eqnarray}
\label{ab-2}
&&
\nonumber
\hskip-.8cm
\P^{RBM}_0\left(H_{a,b}\in du\right)/du
\\
&&
\nonumber
\hskip-.5cm
= \frac{1}{\sqrt{2\pi
    u^3}}\sum_{k=-\infty}^{+\infty}(-1)^k(a-bu-2ka)\,{\rm e}^{-2k^2ab}\,{\rm
  e}^{-\left(a+bu-2ka\right)^2/2u}
  \\
&&
\hskip-.5cm
=:\Delta_{a,b}(u).
\end{eqnarray}
Moreover, in case $b>0$ 
\begin{equation}
\label{nn1}
\P^{RBM}_0\left(\omega_{H_{a,b}}>y\right)=\begin{cases}
\P^{RBM}_0\left(H_{a,b}>\frac
{y-a}{b}\right), & \text {if  $y>a$,}\\ 
1, & \text {if $y\leq a$,}
\end{cases}
\end {equation} 
and for $y>a$
\begin{equation}
\label{nn2}
\P^{RBM}_0\left(\omega_{H_{a,b}}\in dy\right)/dy=
\frac 1 b \Delta_{a,b}(\frac{y-a}b);
\end {equation} 
in case $b<0$ 
\begin{equation}
\label{nn3}
\P^{RBM}_0\left(\omega_{H_{a,b}}<y\right)=\begin{cases}
\P^{RBM}_0\left(H_{a,b}>\frac
{y-a}{b}\right), & \text {if  $y<a$,}\\ 
1, & \text {if $y\geq a$,}
\end{cases}
\end {equation} 
and for $y<a$
\begin{equation}
\label{nn4}
\P^{RBM}_0\left(\omega_{H_{a,b}}\in dy\right)/dy=
\frac 1 {|b|} \Delta_{a,b}(\frac{y-a}b).
\end {equation}
\end{thm}

\begin{proof}
We need to prove the result for $b<0$ (for $b>0$ see \cite{abundo02}). For this, by scaling, it is enough to consider 
$H_{a,-1},$ i.e., we take $b=-1$ (cf. (\ref{scale1})). Using  formula (\ref{ae2})  in
Theorem \ref{abundo2} we have for $u<a$
\begin{eqnarray}
\label{e1}
&&
\nonumber
\hskip-.7cm
\P^{RBM}_0\left(H_{a,-1}>u\rp
=\int_{0}^{a-u}\P^{RBM}_0\lp H_{a,-1}>u\,,\,\omega_u\in
dy\rp
\\
&&
\hskip1.7cm
\nonumber
=\int_{0}^{a-u}\P^{RBM}_0\lp H_{a,-1}>u\,|\,\omega_u=y\rp\, \P^{RBM}_0(\omega_u\in dy)
\\
&&
\hskip1.7cm
\nonumber
=\int_{-(a-u)}^{a-u}\P^{BM}_{0,u,y}\lp -(a-s)<\omega_s <a-s \ \ \forall s\in(0,u)\rp
\\
&&
\hskip6.5cm
\nonumber
\times\, \P^{BM}_0(\omega_u\in dy)
\\
&&
\hskip1.7cm
\nonumber
=\int_{-(a-u)}^{a-u}\P^{BM}_{y,u,0}\lp -(s-(u-a))<\omega_s <s-(u-a) \ \ \forall s\in(0,u)\rp
\\
&&
\hskip6.5cm
\times\, \P^{BM}_0(\omega_u\in dy),
\end{eqnarray}
where, in the last step, we have used the time reversal property of
diffusion bridges, see Salminen \cite{salminen97}, i.e., for Brownian bridge we have 
\begin{equation}
\label{timrev}
\lp\{\omega_t:0\leq t\leq u\}, \P^{BM}_{x,u,y}\rp\rr
\lp\{\omega_{u-t}:0\leq t\leq u\}, \P^{BM}_{y,u,x}\rp.
\end{equation}
Using spatial homogeneity and (\ref{ae2}) in  Theorem \ref{abundo2} we obtain  
\begin{eqnarray*}
&&
\P^{BM}_{y,u,0}\lp -(s-(u-a))<\omega_s <s-(u-a) \ \ \forall s\in(0,u)\rp
\\
&& 
\hskip.5cm
=
\P^{BM}_{0,u,-y}\lp -(s-(u-a))-y<\omega_s <s-(u-a)-y \ \ \forall s\in(0,u)\rp.
\\
&& 
\hskip.5cm
=
\P^{BM}_{0,u,-y}\lp -(a-u+y +s)<\omega_s <a-u-y +s \ \ \forall s\in(0,u)\rp.
\\
&& 
\hskip.5cm
=
G(a-u+y,\frac au;a-u-y,\frac au).
\end{eqnarray*}
To write down the explicit form of the function $G$ we use (\ref{G}), and to stress the dependence of $A_k,$ $B_k,$ $C_k,$ and $D_k$ on $y$ we use notations $A_k(y),$ $B_k(y),$ $C_k(y),$ and $D_k(y),$ respectively. It holds 
\begin{eqnarray*}
&&
A_k(y)=\frac au\lp k^2(a-u-y)+(k-1)^2(a-u+y)+2k(k-1)(a-u)\rp
\\
&&
\hskip1.1cm
=-(2k-1)^2a+\frac 1u\lp(2k-1)^2a^2-(2k-1)ay\rp,
\end{eqnarray*}
\begin{eqnarray*}
&&
B_k(y)=\frac au\lp (k-1)^2(a-u-y)+k^2(a-u+y)+2k(k-1)(a-u)\rp
\\
&&
\hskip1.1cm
=-(2k-1)^2a+\frac 1u\lp(2k-1)^2a^2+(2k-1)ay\rp,
\end{eqnarray*}
\begin{eqnarray*}
&&
C_k(y)=\frac au\lp 2k^2(a-u)+k(k-1)(a-u-y)+k(k+1)(a-u+y)\rp
\\
&&
\hskip1.1cm
=-(2k)^2a+\frac 1u\lp(2k)^2a^2+2kay\rp,
\end{eqnarray*}
and
\begin{eqnarray*}
&&
D_k(y)=\frac au\lp 2k^2(a-u)+k(k+1)(a-u-y)+k(k-1)(a-u+y)\rp
\\
&&
\hskip1.15cm
=-(2k)^2a+\frac 1u\lp(2k)^2a^2-2kay\rp.
\end{eqnarray*}
Now we have, see (\ref{e1}), 
\begin{eqnarray}
\label{e2}
&&
\hskip-.8cm
\P^{RBM}_0\left(H_{a,-1}>u\rp
\\
&&
\hskip-.9cm
\nonumber
=
\int_ {-(a-u)}^{a-u}\lp 1-\sum_{k=1}^\infty\left({\rm e}^{-2A_k(y)}+{\rm e}^{-2B_k(y)}-{\rm
  e}^{-2C_k(y)}-{\rm e}^{-2D_k(y)}\right)\rp\,\varphi_u(y)\,dy,
\end{eqnarray}
where
\[ 
\varphi_u(y):=\frac 1{\sqrt{2\pi u}}\,\e^{-y^2/2u}.
\]
Straightforward integration yields:
\begin{eqnarray*}
&&
\hskip-.5cm
\int_ {-(a-u)}^{a-u}{\rm e}^{-2A_k(y)}\,\varphi_u(y)\,dy
\\
&&
= 
\int_ {-(a-u)}^{a-u}\,\frac 1{\sqrt{2\pi u}}\,\exp\lp
-\frac{y^2}{2u}-\frac 2u\lp(2k-1)^2a^2-(2k-1)ay\rp\rp\,dy 
\\
&&
=
 \e^{2(2k-1)^2a}\int_ {-(a-u)}^{a-u}\,\frac 1{\sqrt{2\pi
     u}}\,\e^{-\lp y-2(2k-1)a\rp^2/2u}\, dy
\\
&&
=\e^{2(2k-1)^2a}\lp\Phi\left(\frac{a-u-2(2k-1)a}{\sqrt u}\right)-\Phi\left(\frac{-(a-u)-2(2k-1)a}{\sqrt u}\right)\rp
\\
&&
=:
{\bf A}(k,u),
\end{eqnarray*}

\begin{eqnarray*}
&&
\hskip-.5cm
\int_ {-(a-u)}^{a-u}{\rm e}^{-2B_k(y)}\,\varphi_u(y)\,dy
\\
&&
=\e^{2(2k-1)^2a}\lp\Phi\left(\frac{a-u+2(2k-1)a}{\sqrt u}\right)-\Phi\left(\frac{-(a-u)+2(2k-1)a}{\sqrt u}\right)\rp
\\
&&
=:{\bf B}(k,u),
\end{eqnarray*}

\begin{eqnarray*}
&&
\hskip-1.8cm
\int_ {-(a-u)}^{a-u}{\rm e}^{-2C_k(y)}\,\varphi_u(y)\,dy
\\
&&
=\e^{2(2k)^2a}\lp\Phi\left(\frac{a-u+4ka}{\sqrt u}\right)-\Phi\left(\frac{-(a-u)+4ka}{\sqrt u}\right)\rp
\\
&&
=:{\bf C}(k,u),
\end{eqnarray*}
and
\begin{eqnarray*}
&&
\hskip-1.8cm
\int_ {-(a-u)}^{a-u}{\rm e}^{-2D_k(y)}\,\varphi_u(y)\,dy
\\
&&
=\e^{2(2k)^2a}\lp\Phi\left(\frac{a-u-4ka}{\sqrt u}\right)-\Phi\left(\frac{-(a-u)-4ka}{\sqrt u}\right)\rp
\\
&&
=:{\bf D}(k,u).
\end{eqnarray*} 
Hence, we have 
\begin{eqnarray}
\label{m1}
&&
\nonumber
\hskip-1cm
\P_0^{RBM}\left(H_{a,-1}>u\right)
\\
&&
\hskip.5cm
={\bf I}(0,u)-\sum_{k=1}^\infty\left({\bf A}(k,u)+{\bf B}(k,u)-{\bf C}(k,u)-{\bf D}(k,u)\right),
\end{eqnarray} 
where 
$$
{\bf I}(0,u)=\Phi\left(\frac{a-u}{\sqrt u}\right)-\Phi\left(-\frac{a-u}{\sqrt u}\right).
$$
Rearranging the terms in (\ref{m1}) yields (\ref{m02}) with $b=-1,$ and differentiating term by term leads to the density given in (\ref{ab-2}).
\end{proof}

\begin{remark}
Taking in (\ref{m02}) $b=0$ and differentiating yields formula (\ref{a00}) for the density of the hitting time of a point for RBM. Clearly, the density given in (\ref{ab-2}) also converges as $b\to 0$ to 
(\ref{a00}).
\end{remark}

\section{Distribution of $H_{a,b}$ for a 3-dimensional Bessel process}
\label{sec4}

In this section we derive an expression for the distribution function of $H_{a,b}$ for a 3-dimensional Bessel process initiated at 0. The probability measures associated with a Bessel bridge from $x$ to $y$ of length $u$ with dimension parameter 3 and a killed (at 0) Brownian bridge  are denoted by $\P^{BES}_{x,u,y}$ and $\P^{KBM}_{x,u,y},$ respectively. Recall that, in fact, 
\begin{equation}
\label{id}
\P^{KBM}_{x,u,y}=\P^{BES}_{x,u,y}
\end{equation}
which follows from BES being Doob's $h$-transform of KBM with $h(x)=x,$ i.e.,
for $\Gamma_t\in\cC_t$ and $y>0$
\begin{equation}
\label{idd}
\P^{BES}_{y}(\Gamma_t)=\E^{BM}_{y}\lp\frac{\omega_{t\wedge H_0}}{y}\,;\,\Gamma_t\rp
\end{equation}
(see McKean \cite{mckean63}, Williams \cite{williams74}, Revuz and Yor \cite{revuzyor01} p. 450 and Borodin and Salminen \cite{borodinsalminen02} p. 75). However, for some arguments below it is good to have a different notation since these bridges are conditioned from different processes.   

We start with a proposition which says that a Brownian bridge from $x>0$ to
$y>0$ conditioned not to hit 0 is identical in law with the killed
Brownian bridge. 

\begin{proposition}
\label{p1}   
For $x,y,u>0$ 
\begin{equation}  
\label{abm}
\P^{KBM}_{x,u,y}=\P^{BM}_{x,u,y}\lp\, \cdot\ |\ \omega_t>0 \quad \forall\, 0\leq t\leq u\rp.
\end{equation}  
\end{proposition}

\begin{proof}
For $0<u_1<\dots<u_n<u$ consider 
\begin{eqnarray*}
&&
\hskip-.8cm
\P^{BM}_{x,u,y}\lp \omega_{u_1}\in dz_1,\dots,\omega_{u_n}\in
dz_n\,;\, \om_t>0\quad \forall\ 0\leq t\leq u\rp
\\
&&
\hskip-.2cm
=  
\P^{BM}_{x}\lp \omega_{u_1}\in dz_1,\dots,\omega_{u_n}\in
dz_n, \om_u\in dy\,;\, H_0>u\rp/\P^{BM}_{x}\lp\om_u\in dy\rp
\\
&&
\hskip-.2cm
=  
\P^{KBM}_{x}\lp \omega_{u_1}\in dz_1,\dots,\omega_{u_n}\in
dz_n, \om_u\in dy\rp/\P^{BM}_{x}\lp\om_u\in dy\rp,
\end{eqnarray*}
where $H_0:=\inf\{t: \om_t=0\}.$ Similarly, 
\begin{eqnarray*}
&&
\P^{BM}_{x,u,y}\lp  \om_t>0\quad \forall\ 0\leq t\leq u\rp
=  
\P^{KBM}_{x}\lp \om_u\in dy\rp/\P^{BM}_{x}\lp\om_u\in dy\rp.
\end{eqnarray*}
Consequently,
\begin{eqnarray*}
&&
\hskip-.8cm
\P^{BM}_{x,u,y}\lp \omega_{u_1}\in dz_1,\dots,\omega_{u_n}\in
dz_n\,|\, \om_t>0\quad \forall\ 0\leq t\leq u\rp
\\
&&
\hskip-.2cm
=  
\P^{KBM}_{x}\lp \omega_{u_1}\in dz_1,\dots,\omega_{u_n}\in
dz_n, \om_u\in dy\rp/
\P^{KBM}_{x}\lp \om_u\in dy\rp
\\
&&
\hskip-.2cm
=  
\P^{KBM}_{x,u,y}\lp \omega_{u_1}\in dz_1,\dots,\omega_{u_n}\in
dz_n\rp.
\end{eqnarray*}
Then, the monotone class theorem implies the stated identity. 
\end{proof}

\begin{thm}
\label{t4}
For $a>0$  and $b\in\R$ such that $u<-(a/b)$ for $b<0$ (and there is no constraint if $b>0$) it holds 
\begin{eqnarray}
\label{m11}
&&
\nonumber
\P^{BES}_0\left(H_{a,b}>u\right)
\\
&&
\nonumber
\hskip1cm
=2\sum_{k=-\infty}^\infty
\e^{-2k^2ab}\,\Big[\left(1-4k^2ab\right)
\left(\Phi\left(\frac{a+bu-2ka}{\sqrt{u}}\right)-\Phi\left(-\frac{2ka}{\sqrt{ u}}\right)\right)
\\
&&
\hskip2cm
- \frac{2kbu-a-bu}{\sqrt{u}}\, \varphi\lp\frac{a+bu-2ka}{\sqrt{u}}\right)\Big],
\end{eqnarray}
where $\varphi$ is the standard normal density.
\end{thm}
\begin{proof}
Consider first the case $b<0$ and $u<-(a/b).$ Without loss of generality
we may take $b=-1.$ Proceeding as in the proof
of Theorem \ref{t3} (cf. (\ref{e1})) we have
\begin{eqnarray}
\label{m12}
&&
\nonumber
\hskip-.7cm
\P^{BES}_0\left(H_{a,-1}>u\rp
=\int_{0}^{a-u}\P^{BES}_{y,u,0}\lp \om_s <s+a-u \ \forall s\in(0,u)\rp
\\
&&
\hskip6.5cm
\times\, \P^{BES}_0(\om_u\in dy).
\end{eqnarray}
By (\ref{id}) and (\ref{abm}), 
\begin{eqnarray}
\label{m13}
&&
\nonumber
\hskip-.7cm
\P^{BES}_{y,u,0}\lp \om_s <s+a-u \ \forall s\in(0,u)\rp
\\
&&
\nonumber
\hskip1.5cm
=
\P^{KBM}_{y,u,0}\lp \om_s <s+a-u \ \forall s\in(0,u)\rp
\\
&&
\hskip1.5cm
=
\lim_{x\downarrow 0}\frac{\P^{BM}_{y,u,x}\lp 0<\om_s <s+a-u
  \ \forall s\in(0,u)\rp}
{\P^{BM}_{y,u,x}\lp \om_s>0 \ \forall s\in(0,u)\rp}.
\end{eqnarray}
From Theorem \ref{abundo2} formula (\ref{ae2}) we have for $x,y>0$
\begin{eqnarray}
\label{m14}
&&
\nonumber
\P^{BM}_{y,u,x}\lp 0<\om_s <s+a-u
  \ \forall s\in(0,u)\rp
\\
&&
\nonumber
\hskip1.5cm
= 
 \P^{BM}_{0,u,x-y}\lp -y<\om_s <s+a-u-y
  \ \forall s\in(0,u)\rp
\\
&&
\hskip1.5cm
=G\lp y,\frac xu,a-u-y,\frac{a-x}u\rp ,
\end{eqnarray}
and 
\begin{eqnarray}
\label{m15}
&&
\nonumber
\P^{BM}_{y,u,x}\lp \om_s >0
  \ \forall s\in(0,u)\rp
\\
&&
\nonumber
\hskip1.5cm
= 
\lim_{N\to\infty} \P^{BM}_{0,u,x-y}\lp -y<\om_s <N-y
  \ \forall s\in(0,u)\rp
\\
&&
\hskip1.5cm
=
\lim_{N\to\infty}G\lp y,\frac xu;N-y,\frac{N-x}u\rp .
\end{eqnarray}
From (\ref{m15}) we may recover the well known formula
(cf. \cite{borodinsalminen02} p. 174) 
\begin{equation}
\label{m16}
\P^{BM}_{y,u,x}\lp \om_s >0
  \ \forall s\in(0,u)\rp
=1-\e^{-2xy/u}.
\end{equation}

To write down the explicit form of the function $G$ in (\ref{m14}) we use (\ref{G}). The functions $A_k, B_k, C_k$ and $D_k$ are now denoted by $A_k(x,y), B_k(x,y), C_k(x,y)$ 
and $D_k(x,y)$ and it holds 
\begin{eqnarray*}
&&
\hskip-1cm
A_k(x,y)=\frac 1u\big[ k^2(a-u-y)(a-x)+(k-1)^2xy
\\
&&
\hskip2cm
+k(k-1)((a-u-y)x+(a-x)y)\big],
\end{eqnarray*}
\begin{eqnarray*}
&&
\hskip-1cm
B_k(x,y)=\frac 1u\big[ (k-1)^2(a-u-y)(a-x)+k^2xy
\\
&&
\hskip2cm
+k(k-1)\lp(a-u-y)x+(a-x)y\rp\big],
\end{eqnarray*}
\begin{eqnarray*}
&&
\hskip-.5cm
C_k(x,y)=\frac 1u\big[ k^2\lp(a-u-y)(a-x)+xy\rp
\\
&&
\hskip2cm
+k(k-1)(a-u-y)x+k(k+1)(a-x)y\big],
\end{eqnarray*}
and
\begin{eqnarray*}
&&
\hskip-.5cm
D_k(x,y)=\frac 1u\big[ k^2\lp(a-u-y)(a-x)+xy\rp
\\
&&
\hskip2cm
+k(k+1)(a-u-y)x+k(k-1)(a-x)y\big].
\end{eqnarray*}
Moreover,
$$
A_k(0,y)=D_k(0,y)=\frac au\lp k^2(a-u)-ky\rp,
$$
$$
B_{k+1}(0,y)=C_k(0,y)=\frac au\lp k^2(a-u)+ky\rp,
$$
$$
A'_k:=\frac{\partial A_k}{\partial x}(x,y)
=-k\frac {a-u}u+\frac yu,
$$
$$
B'_k:=\frac{\partial B_k}{\partial x}(x,y)
=(k-1)\frac {a-u}u+\frac yu,
$$
and
$$
C'_k:=\frac{\partial C_k}{\partial x}(x,y)=-D'_k:=
-\frac{\partial D_k}{\partial x}(x,y)
=-k\frac{a-u}{u}.
$$
We calculate the limit in (\ref{m13}) using l'Hospital's rule:
\begin{eqnarray}
\label{m17}
&&
\nonumber
\P^{BES}_{y,u,0}\lp\om_s<s+a-u\quad \forall\ s\in(0,u)\rp
\\
&&
\nonumber
\hskip1cm
=
\lim_{x\downarrow 0} \frac{{ 1-\sum_{k=1}^\infty\left({\rm e}^{-2A_k}+{\rm e}^{-2B_k}-{\rm
  e}^{-2C_k}-{\rm e}^{-2D_k}\right)}}{1-\e^{-2xy/u}}
\\
&&
\nonumber
\hskip1cm
=
\frac{ B'_1\e^{-2B_1(0)}+\sum_{k=1}^\infty\left(\lp A'_k-D'_k\rp{\rm
    e}^{-2A_k(0)}+\lp B'_{k+1}-C'_k\rp{\rm e}^{-2C_k(0)}\right)}{y/u}
\\
&&
\nonumber
\hskip1cm
=
1+\frac{2(a-u)}{y}\,\sum_{k=1}^\infty k\, 
{\rm  e}^{2k^2a}{\rm e}^{-2k^2a^2/u}\lp {\rm  e}^{-2kay/u}- {\rm  e}^{2kay/u}\rp
\\
&&
\nonumber
\hskip4cm
+
\sum_{k=1}^\infty 
{\rm  e}^{2k^2a}{\rm e}^{-2k^2a^2/u}\lp {\rm  e}^{-2kay/u}+ {\rm  e}^{2kay/u}\rp
\\
&&
\nonumber
\hskip1cm
=
\frac{2(a-u)}{y}\,\sum_{k=-\infty}^\infty k\, 
{\rm  e}^{2k^2a}{\rm e}^{-2k^2a^2/u}{\rm  e}^{-2kay/u}
\\
&&
\nonumber
\hskip4cm
+
\sum_{k=-\infty}^\infty 
{\rm  e}^{2k^2a}{\rm e}^{-2k^2a^2/u}{\rm  e}^{-2kay/u}
\\
&&
\hskip1cm
=
\sum_{k=-\infty}^\infty \lp 1-\frac {2k(a-u)}{y}\rp 
{\rm  e}^{2k^2a}{\rm e}^{-2k^2a^2/u}{\rm  e}^{2kay/u},
\end{eqnarray}
where, in the second step, $B_1(0):=B_1(0,y)=0, A_k(0):=A_k(0,y),$ and $C_k(0):=C_k(0,y).$
Now we can proceed by evaluating the integral given in
(\ref{m12}). Recall that  
$$
 \P^{BES}_0(\om_u\in dy)=\frac{2y^2}{\sqrt{2\pi u^3}}\,\e^{-y^2/2u}\,dy,
$$
and consider the integrals
\begin{eqnarray*}
&&
\nonumber
\hskip-1cm
I_1:=\int_0^{a-u} {\rm e}^{-2k^2a^2/u}{\rm  e}^{2kay/u}\, 
\frac{2y^2}{\sqrt{2\pi u^3}}\,\e^{-y^2/2u}\,dy
\\
&&
\nonumber
\hskip-.5cm
=
\int_0^{a-u}
\frac{2y^2}{\sqrt{2\pi u^3}}\,\e^{-(y-2ka)^2/2u}\,dy,
\end{eqnarray*}
and
\begin{eqnarray*}
&&
\nonumber
I_2:=\int_0^{a-u} \frac {2k(a-u)}{y}\,{\rm e}^{-2k^2a^2/u}{\rm  e}^{2kay/u}\, 
\frac{2y^2}{\sqrt{2\pi u^3}}\,\e^{-y^2/2u}\,dy
\\
&&
\nonumber
\hskip.5cm
= 2k(a-u)
\int_0^{a-u}
\frac{2y}{\sqrt{2\pi u^3}}\,\e^{-(y-2ka)^2/2u}\,dy.
\end{eqnarray*}
Writing 
\begin{eqnarray*}
&&
\hskip-1cm
I_1= 
\int_0^{a-u}
\frac{2y(y-2ka)}{\sqrt{2\pi u^3}}\,\e^{-(y-2ka)^2/2u}\,dy
\\
&&
\nonumber
\hskip2cm
+2ka
\int_0^{a-u}
\frac{2y}{\sqrt{2\pi u^3}}\,\e^{-(y-2ka)^2/2u}\,dy
\\
&&
\nonumber
\hskip-.5cm
=: I_{11}+I_{12},
\end{eqnarray*}
and
\begin{eqnarray*}
&&
\hskip-1cm
I_2= 2ka
\int_0^{a-u}
\frac{2y}{\sqrt{2\pi u^3}}\,\e^{-(y-2ka)^2/2u}\,dy
\\
&&
\nonumber
\hskip2cm
- 
2k
\int_0^{a-u}
\frac{2y}{\sqrt{2\pi u}}\,\e^{-(y-2ka)^2/2u}\,dy
\\
&&
\nonumber
\hskip-.5cm
=: I_{21}-I_{22}
\end{eqnarray*}
it is seen that $I_{12}=I_{21}.$ Straightforward calculations yield
$$
I_{11}=-2\,\frac{a-u}{\sqrt{u}}\,\varphi\lp\frac{a-u-2ka}{\sqrt u}\rp+2\lp
\Phi\lp\frac{a-u-2ka}{\sqrt u}\rp
-\Phi\lp-\frac{2ka}{\sqrt u}\rp\rp.
$$
Furthermore,
\begin{eqnarray*}
&&
\hskip-1cm
I_{22}=2k\sqrt{u} \,
\int_0^{a-u}
\frac{2(y-2ka)}{\sqrt{2\pi }\,u}\,\e^{-(y-2ka)^2/2u}\,dy
\\
&&
\nonumber
\hskip2cm
+ 2(2k)^2a\,\int_0^{a-u}
\frac{1}{\sqrt{2\pi u }}\,\e^{-(y-2ka)^2/2u}\,dy
\\
&&
\nonumber
\hskip-.5cm
=4k \sqrt{u}\, \lp \varphi\lp\frac{2ka}{\sqrt u}\rp-
\varphi\lp\frac{a-u-2ka}{\sqrt u}\rp\rp
\\
&&
\nonumber
\hskip2cm
+ 2(2k)^2a
\lp 
\Phi\lp\frac{a-u-2ka}{\sqrt u}\rp
-\Phi\lp-\frac{2ka}{\sqrt u}\rp\rp.
\end{eqnarray*}
Consequently,
\begin{eqnarray*}
&&
\hskip-1cm
\P^{BES}_0\left(H_{a,-1}>u\right)
=
\sum_{k=-\infty}^\infty {\rm  e}^{2k^2a}\lp I_{1}-I_{2}\rp
\\
&&
\nonumber
\hskip2.15cm
=
\sum_{k=-\infty}^\infty {\rm  e}^{2k^2a}\lp I_{11}+I_{22}\rp
\\
&&
\nonumber
\hskip2.15cm
=
2\,\sum_{k=-\infty}^\infty {\rm  e}^{2k^2a}\Big[\lp\frac{2ku -a+u}{\sqrt
    u}\rp\varphi\lp\frac{a-u-2ka}{\sqrt u}\rp 
\\
&&
\nonumber
\hskip2.5cm
+\lp 1+(2k)^2a\rp \lp
\Phi\lp\frac{a-u-2ka}{\sqrt u}\rp
-\Phi\lp-\frac{2ka}{\sqrt u}\rp\rp\Big], 
\end{eqnarray*}
which is the claimed formula for $b=-1,$ and, by scaling (see (\ref{scale1})) we
may extend this for general $b<0.$

Let now $b=1,$ and consider  
\begin{eqnarray}
\label{m18}
&&
\nonumber
\hskip-.7cm
\P^{BES}_0\left(H_{a,1}>u\rp
=\int_{0}^{a+u}\P^{BES}_{0,u,y}\lp \om_s <s+a \ \forall s\in(0,u)\rp
\\
&&
\hskip6.5cm
\times\, \P^{BES}_0(\om_u\in dy).
\end{eqnarray}
We have (see (\ref{m13}))
\begin{eqnarray}
\label{m19}
&&
\nonumber
\hskip-.7cm
\P^{BES}_{0,u,y}\lp \om_s <s+a \ \forall s\in(0,u)\rp
\\
&&
\hskip1.5cm
=
\lim_{x\downarrow 0}\frac{\P^{BM}_{x,u,y}\lp 0<\om_s <s+a
  \ \forall s\in(0,u)\rp}
{\P^{BM}_{x,u,y}\lp \om_s>0 \ \forall s\in(0,u)\rp}.
\end{eqnarray}
Using (\ref{ae2}) gives (see (\ref{m14}))
\begin{eqnarray*}
&&
\P^{BM}_{x,u,y}\lp 0<\om_s <s+a
  \ \forall s\in(0,u)\rp
\\
&&
\hskip2cm
=
\P^{BM}_{0,u,y-x}\lp -x<\om_s <s+a-x
  \ \forall s\in(0,u)\rp
\\
&&
\hskip2cm
=G\lp x,\frac yu;a-x,1+\frac{a-y}u\rp.
\end{eqnarray*}
To find the explicit form of the function $G,$ let $\hat A_k(x,y), \hat B_k(x,y), \hat C_k(x,y),$ and $\hat D_k(x,y)$ denote now the terms $A_k, B_k, C_k,$ and $D_k$ in (\ref{G}). It holds
\begin{eqnarray*}
&&
\hskip-1cm
\hat A_k(x,y)=\frac 1u\big[ k^2(a+u-y)(a-x)+(k-1)^2xy
\\
&&
\hskip2cm
+k(k-1)((a+u-y)x+(a-x)y)\big],
\end{eqnarray*}
\begin{eqnarray*}
&&
\hskip-1cm
\hat B_k(x,y)=\frac 1u\big[ (k-1)^2(a+u-y)(a-x)+k^2xy
\\
&&
\hskip2cm
+k(k-1)\lp(a+u-y)x+(a-x)y\rp\big],
\end{eqnarray*}
\begin{eqnarray*}
&&
\hskip-.5cm
\hat C_k(x,y)=\frac 1u\big[ k^2\lp(a+u-y)(a-x)+xy\rp
\\
&&
\hskip2cm
+k(k-1)(a-x)y+k(k+1)(a+u-y)x\big],
\end{eqnarray*}
and
\begin{eqnarray*}
&&
\hskip-.5cm
\hat D_k(x,y)=\frac 1u\big[ k^2\lp(a+u-y)(a-x)+xy\rp
\\
&&
\hskip2cm
+k(k+1)(a-x)y+k(k-1)(a+u-y)x\big].
\end{eqnarray*}
Notice that, e.g., $\hat A_k(x,y)$ can be obtained from $A_k(x,y)$ by substituting $-u$ instead of $u$ and multiplying by $-1.$ 

Moreover,
$$
\hat A_k(0,y)=\hat C_k(0,y)=\frac au\lp k^2(a+u)-ky\rp,
$$
$$
\hat B_{k+1}(0,y)=\hat D_k(0,y)=\frac au\lp k^2(a+u)+ky\rp,
$$
$$
\frac{\partial \hat A_k}{\partial x}(x,y)
=-k\frac {a+u}u+\frac yu,
$$
$$
\frac{\partial B_k}{\partial x}
=(k-1)\frac {a+u}u+\frac yu,
$$
and
$$
\frac{\partial \hat C_k}{\partial x}(x,y)=
-\frac{\partial \hat D_k}{\partial x}(x,y)
=k\frac{a+u}{u}.
$$
Consequently, as in (\ref{m17}),
\begin{eqnarray}
\label{m20}
&&
\nonumber
\P^{BES}_{0,u,y}\lp\om_t<a+t\quad \forall\ t\in(0,u)\rp
\\
&&
\nonumber
\hskip1cm
=
1+\frac{2(a+u)}{y}\,\sum_{k=1}^\infty k\, 
{\rm  e}^{-2k^2a}{\rm e}^{-2k^2a^2/u}\lp {\rm  e}^{-2kay/u}- {\rm  e}^{2kay/u}\rp
\\
&&
\nonumber
\hskip4cm
+
\sum_{k=1}^\infty 
{\rm  e}^{-2k^2a}{\rm e}^{-2k^2a^2/u}\lp {\rm  e}^{2kay/u}+ {\rm  e}^{-2kay/u}\rp
\\
&&
\hskip1cm
=
\sum_{k=-\infty}^\infty \lp 1-\frac {2k(a+u)}{y}\rp 
{\rm  e}^{-2k^2a}{\rm e}^{-2k^2a^2/u}{\rm  e}^{2kay/u},
\end{eqnarray}
and, further,
\begin{eqnarray*}
&&
\hskip-1cm
\P^{BES}_0\left(H_{a,1}>u\right)
=
2\,\sum_{k=-\infty}^\infty {\rm  e}^{-2k^2a}\Big[\lp\frac{-2ku -a-u}{\sqrt
    u}\rp\varphi\lp\frac{a+u-2ka}{\sqrt u}\rp 
\\
&&
\nonumber
\hskip2.5cm
+\lp 1-(2k)^2a\rp \lp
\Phi\lp\frac{a+u-2ka}{\sqrt u}\rp
-\Phi\lp-\frac{2ka}{\sqrt u}\rp\rp\Big] 
\end{eqnarray*} 
and, by scaling, see (\ref{scale1}), we
may extend this for general $b\geq 0.$ Since formulas
(\ref{nn1})-(\ref{nn4}) follow easily, the proof of the theorem is now
complete.  
\end{proof}
\begin{remark}
It is easily checked that taking $b=0$ in (\ref{m11}) and differentiating yields formula 
(\ref{a11}). Moreover, letting $y\downarrow 0$ in (\ref{m20}) and (\ref{m17}), respectively, 
we obtain the following formula with $b=\pm 1$ (if $b<0$ then the formula is valid for $u<a/(-b)$)
\begin{eqnarray}
\label{1e}
&&
\nonumber
\hskip-1cm
\P^{BES}_{0,u,0}\lp\om_t<a+bt\quad \forall\ t\in(0,u)\rp
\\
&&
\hskip1cm
=1+
2\sum_{k=1}^\infty 
\lp 1-4k^2\frac{a(a+b u)}{u}\rp {\rm  e}^{-2k^2a(a+b u)/u}.
\end{eqnarray}
The case with general $b$ can be proved, e.g., via scaling. Notice that this probability is a function of $a(a+bu)/u$ only. In fact, see Theorem  \ref{t30} 
$$
\P^{BES}_{0,u,0}\lp\om_t<a+bt\quad \forall\ t\in(0,u)\rp
=
\P^{BES}_{0,1,0}\lp \sup_{t\leq 1} \omega_t<\sqrt{\frac{a(a+bu)}{u}}\rp
$$
Putting in (\ref{1e}) $b=0$ gives the well known formula for the distribution of the maximum of a Brownian excursion found in Chung \cite{chung76}, see also Kennedy \cite{kennedy76b}, Durrett and Iglehart \cite{durrettiglehart77}, Biane and Yor \cite{bianeyor87} and Pitman and Yor \cite{pitmanyor99b}.
The distribution of the maximum of a general Bessel process is discussed in Theorem \ref{t30}.
\end{remark}

Next we consider the distribution of $H_{a,b}$ for a three dimensional Bessel process initiated above the line $t\mapsto a+bt.$ In this case, the density can be found using (local) absolute continuity of the  probability measure induced by the Bessel process with respect to the probability measure induced by a killed Brownian motion. 

\begin{thm}
\label{t41}
For $x>a>0$  and $b>0$ it holds 
\begin{eqnarray}
\label{m111}
&&
\P^{BES}_x\left(H_{a,b}\in dt\right)=\frac{(a+bt)(x-a)}{tx\sqrt{2\pi t}}
{\rm e}^{-(a-x-bt)^2/2t}\,dt.
\end{eqnarray}
In particular,
\begin{eqnarray}
\label{m112}
&&
\P^{BES}_0\left(H_{0,b}\in dt\right)=\frac{b}{\sqrt{2\pi t}}
{\rm e}^{-b^2t/2}\,dt.
\end{eqnarray}
\end{thm}
\begin{proof} Since $H_{a,b}<H_0$ for $b>0,$ we obtain, using absolute continuity (\ref{idd}),
\begin{eqnarray*}
&&
\P^{BES}_x\left(H_{a,b}\in dt\right)=
\E^{BM}_x\left(\frac{\omega_{H_{a,b}}}{x}\,;\,H_{a,b}\in dt\right)
\\
&&
\hskip3cm
=\frac{a+bt}{x}\P^{BM}_x\left(H_{a,b}\in dt\right)
\\
&&
\hskip3cm
=\frac{(a+bt)(x-a)}{tx\sqrt{2\pi t}}
{\rm e}^{-(a-x-bt)^2/2t}\, dt
\end{eqnarray*}
by the well known formula for the density of the hitting time of straight lines for Brownian motion (see, e.g., \cite{borodinsalminen02} p. 295). For formula (\ref{m112}), recall that 0 is an entrance boundary point for the three dimensional Bessel process and, hence, (\ref{m112}) is obtained by taking in (\ref{m111}) $a=0$ and letting $x\to 0.$
\end{proof}

\begin{remark}
\label{piyo}
For Bessel processes with dimension parameter $\delta>2$ we have that $H_{0,b}>0$ $\P^{BES}_0$-a.s. for all $b>0$ (see Shiga and Watanabe \cite{shigawatanabe73} Theorem 3.3). 
On the other hand, it is also well known that $\omega_t/t\to 0$ $\P_x^{BES(\delta)}$-a.s. for all $x\geq 0.$ Consequently, $0<H_{a,b}<\infty$ a.s. for all $a\geq 0$ and $b>0.$ From Pitman and Yor \cite{pitmanyor81} Section 8 p. 332 the following result for hitting times of lines going through the origin holds
\begin{equation}
\label{py1}
\P_x^\delta\left(H_{0,b}\in dt\right)/dt =(bt)^\nu\,{\rm e}^{-\frac{1}{2}\left(b^2t+x^2/t\right)}/{\left(2tx^\nu\,K_\nu(xb)\right)}, \qquad x>0, 
\end{equation}
and
\begin{equation}
\label{py2}
\P_0^\delta\left(H_{0,b}\in dt\right)/dt ={\left(\dfrac{b^2}{2}\right)^\nu\,t^{\nu-1}{\rm e}^{-\frac{b^2}{2}t}}/{\Gamma(\nu)}, 
\end{equation}
where $\nu=(\delta-2)/2$ and $K_\nu$ is the modified Bessel function of the second kind (see Abramowitz and Stegun \cite{abramowitzstegun70} p. 374).
Since $K_{1/2}(x)=\sqrt{\pi}{\rm e}^{-x}/\sqrt{2x}$ and $\Gamma(1/2)=\sqrt{\pi}$ it is seen that 
formulas (\ref{py1}) and (\ref{py2}) coincide with  (\ref{m111}) and (\ref{m112}) in case $\delta=3$ and $a=0.$
\end{remark}

Alili and Patie derive in \cite{alilipatie10} an expression for the
density of hitting time of a Bessel process to a straight line.  We
wish to compare their formulas in case of RBM and BES with the
formulas presented in Theorems \ref{t3} and \ref{t4}. The next result
is Theorem 5.1 in \cite{alilipatie10} (we have, however, corrected a
misprint in \cite{alilipatie10} by putting in  minus signs in front of $b$ and $b^2$ in the
exponent of the first exponential). 

\begin{thm}
\label{alpa}
For a Bessel process of dimension parameter $\delta>0$ and initial state $0\leq x\leq a$ it holds 
\begin{eqnarray}
\label{ap}
&&
\nonumber
\hskip-1.5cm
\P^{BES(\delta)}_x\left(H_{a,b}\in dt\right)/dt =\frac{{\rm e}^{\,-\frac{b}{2a}(a^2-x^2)-
\frac{b^2}{2}\,t}}{(1+\frac{b}{a}\,t)^{\nu+2}}
\\
&&
\hskip2cm
\times 
\sum_{k=1}^\infty 
\frac{x^{-\nu}\,j_{\nu,k}\,J_\nu\left(j_{\nu,k}\frac{x}{a} \right) }
{a^{2-\nu}\,J_{\nu+1}\left(j_{\nu,k} \right)} 
\,{\rm e}^{\,-j_{\nu,k}^2\frac{t}{2a(a+bt)}},
\end{eqnarray}
where $b\in\R,$ $\nu=(\delta-2)/2$ and $J_{\nu}$ is the Bessel function of the first kind and  
$j_{\nu,k}$ is the ordered sequence of the positive zeros of $J_{\nu}.$ In particular,
for RBM starting from 0, i.e., $\nu=-1/2$ and
\[ 
J_{-1/2}(z)=\sqrt{\frac{2}{\pi z}}\,\cos z,\quad 
J_{1/2}(z)=\sqrt{\frac{2}{\pi z}}\,\sin z,
 \]
it holds
\begin{eqnarray}
\label{rbma}
&&
\nonumber
\hskip-1cm
\P^{RBM}_0\left(H_{a,b}\in dt\right)/dt =\frac{{\rm e}^{\,-\frac{ba}{2}-\frac{b^2t}2}}
{\sqrt a(a+bt)^{3/2}}
\\
&&
\nonumber
\hskip2cm
\times 
\sum_{k=1}^\infty (-1)^{k+1}(k-\frac 12)\pi 
\,\exp\lp{-((k-\frac{1}{2})\pi)^2\,\frac{t}{2a(a+bt)}}\rp
\\
&&
\hskip2.5cm
=:D_{a,b}(t).
\end{eqnarray}
\end{thm}
Comparing the results in Theorems \ref{alpa} and \ref{t3}, we obtain, with our notations in   (\ref{rbma}) and (\ref{ab-2}), that
\begin{equation}
\label{problem}
D_{a,b}(t)=\Delta_{a,b}(t).
\end{equation}
This is a Poisson summation type identity to which we would like to devote attention in a forthcoming work \cite{salminenyor11}. For the moment, we only refer to Feller 
\cite{feller71} and  Biane, Pitman and Yor \cite{bianepitmanyor01} for some related discussion.

\vskip1cm

\hskip3cm
Figure 1 about here.

\vskip1cm

\vskip1cm

\hskip3cm
Figure 2 about here.

\vskip1cm

\begin{figure}[h]
\begin{center}
\epsfig{file=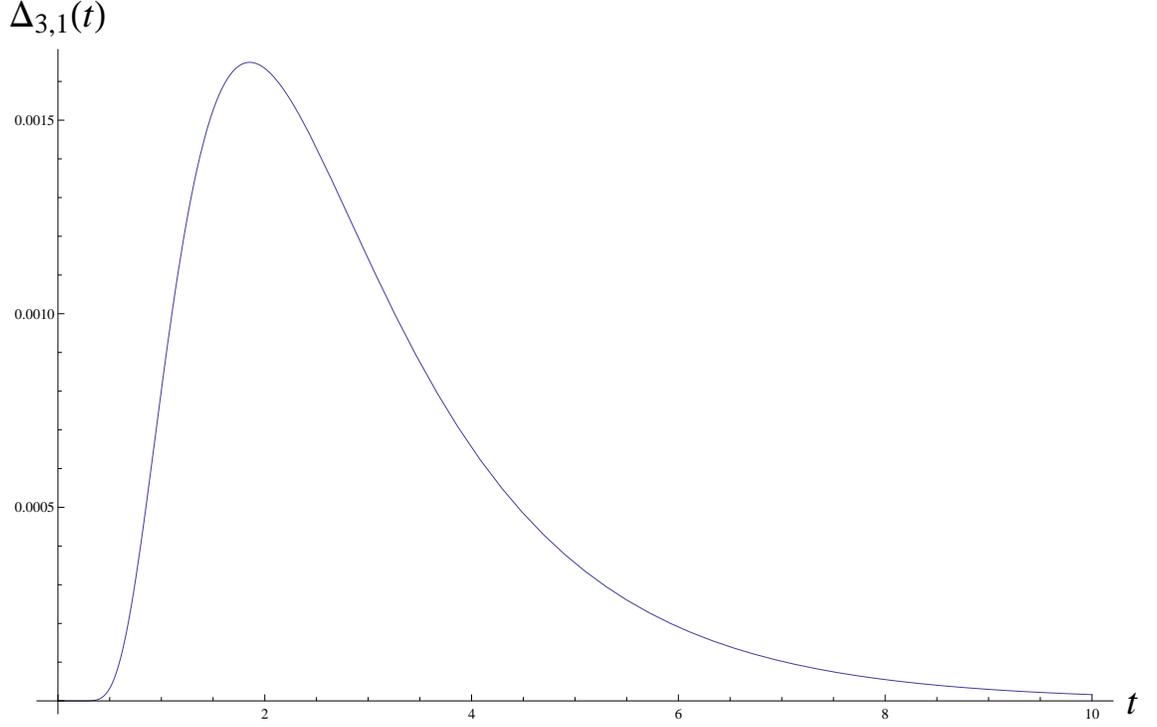,width=15cm}
\end{center}
\caption{\em The density of the distribution of the first 
hitting time of $t\mapsto 3+t$ by RBM started from 0.}
\end{figure}
\begin{figure}[h]
\begin{center}
\epsfig{file=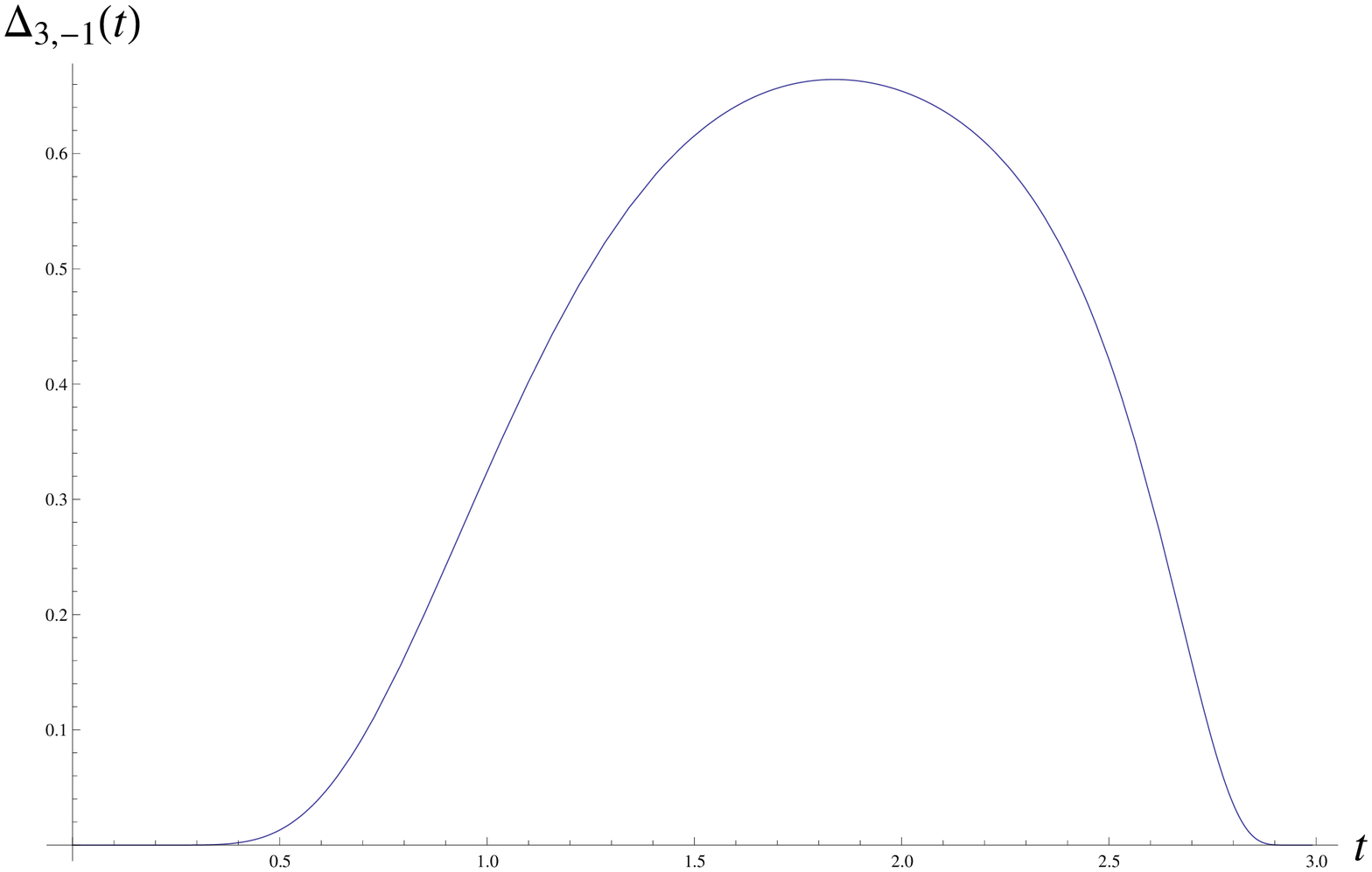,width=15cm}
\end{center}
\caption{\em The density of the distribution of the first 
hitting time of $t\mapsto 3-t$ by RBM started from 0.}
\end{figure}

\section{Results for general Bessel processes}
\label{sec5}

In this section we identify   first the probability that a Bessel bridge crosses an
affine boundary with the probability that the maximum of the Bessel
bridge is below some value. For brevity, we change the notation used above and let $\P^{(\delta)}_x$ denote the probability of a Bessel process with dimension parameter $\delta$ initiated from $x\geq 0.$ Moreover, $\P^{(\delta)}_{0,u,0}$ denotes the probability measure of a Bessel bridge from 0 to 0 of length $u$. For $0<\delta<2$ the Bessel process is
assumed to be instantly reflecting at 0. Secondly, we show that the probability for a Bessel process to stay below the straight line $t\mapsto a+t$ equals to the probability that the maximum of the  corresponding Bessel bridge from 0 to 0 of length 1 is less than $\sqrt a$.

\begin{thm}
\label{t30}
For $a>0,$ $u>0,$ and $b<a/u$ 
\begin{equation} 
\label{bes00}
\P^{(\delta)}_{0,u,0}\left(\sup_{t\leq u}\{\omega_t+bt\}<a\right)=
\P^{(\delta)}_{0,u,0}\left(\sup_{t\leq
  u}\{\omega_t\}<\sqrt{a(a-bu)}\right).
\end{equation}
\end{thm}

\begin{proof}
Consider first the case $b>0$ and let $\Delta_l$ denote the left hand side of
(\ref{bes00}). We have
\begin{eqnarray*}
&&
\Delta_l=\P^{(\delta)}_{0,u,0}\left(\omega_t\leq a-bt\quad \forall 0<t<u\rp
\\
&&
\hskip.5cm
=\P^{(\delta)}_{0,u,0}\left(\omega_t\leq a-bu+bt\quad \forall 0<t<u\rp
\\
&&
\hskip.5cm
=\P^{(\delta)}_0\left(\omega_t\leq a-bu+bt\quad \forall 0<t<u\,|\, \omega_u=0\rp,
\end{eqnarray*}
where in the second step the time reversal property of diffusion
bridges is used (cf. (\ref{timrev})).

By the time inversion of Bessel processes we obtain 
\begin{eqnarray*}
&&
\Delta_l=\P^{(\delta)}_0\left(t\,\omega_{1/t}\leq a-bu+bt\quad \forall\ 0<t<u\,|\,u\, \omega_{1/u}=0\rp,
\\
&&
\hskip.5cm
=\P^{(\delta)}_0\left(\omega_{1/t}\leq (a-bu)\frac 1t+b\quad \forall\ \frac 1t>\frac 1u\,|\, \omega_{1/u}=0\rp,
\\
&&
\hskip.5cm
=\P^{(\delta)}_0\left(\omega_{s}\leq (a-bu)(s+\frac 1u)+b\quad \forall\ s>0\rp
\\
&&
\hskip.5cm
=\P^{(\delta)}_0\left(\omega_{s}\leq (a-bu)s+\frac au\quad \forall\ s>0\rp.
\end{eqnarray*}
Letting $\Delta_r$ denote the right hand side of
(\ref{bes00}) we write
\begin{eqnarray}
\label{bes01}
&&
\nonumber
\Delta_r=\P^{(\delta)}_{0,u,0}0\left(\omega_t\leq \sqrt{a(a-bu)}\quad \forall\ 0<t<u\rp
\\
&&
\nonumber
\hskip.5cm
=\P^{(\delta)}_0\left(t\,\omega_{1/t}\leq  \sqrt{a(a-bu)}\quad \forall\ \frac 1t>\frac 1u\,|\,u\,\omega_{1/u}=0\rp
\\
&&
\nonumber
\hskip.5cm
=\P^{(\delta)}_0\left(\omega_s\leq  \sqrt{a(a-bu)}\,s\quad \forall\ s>\frac 1u\,|\, \omega_{1/u}=0\rp
\\
&&
\nonumber
\hskip.5cm
=\P^{(\delta)}_0\left(\omega_s\leq  \sqrt{a(a-bu)}\lp s+\frac 1u\rp\quad \forall\ s>0\rp
\\
&&
\hskip.5cm
=\P^{(\delta)}_0\left(\omega_s\leq  \sqrt{a(a-bu)}\lp s+\frac 1u\rp\quad \forall\ s>0\rp.
\end{eqnarray}
Recall that Bessel processes enjoy the  Brownian scaling property: 

\noindent
for all $c>0$ under 
$\P^{(\delta)}_0$
\begin{equation}
\label{bes02}
\{\sqrt c\, \omega_{t/c}:t\geq 0\}\,\rr\,\{\omega_{t}:t\geq 0\}.
\end{equation}
Applying this with $c=(a-bu)/a$ in (\ref{bes01}) yields
\begin{eqnarray*}
&&
\Delta_r=
\P^{(\delta)}_0\left(\omega_{as/(a-bu)}\leq  as+\frac au \quad \forall\ s>0\rp
\\
&&
\hskip.5cm
=
\P^{(\delta)}_0\left(\omega_{t}\leq  (a-bu)t+\frac au \quad \forall\ t>0\rp,
\end{eqnarray*}
which equals $\Delta_l.$ 

For $b<0$ the proof is slightly simpler since we do not use
time reversal. Indeed, letting again $\Delta_l$ and $\Delta_r$ denote the
right and left hand side of  (\ref{bes00}), respectively, we have
by time inversion
\begin{eqnarray}
\label{bes021}
&&
\nonumber
\Delta_l=\P^{(\delta)}_{0,u,0}\left(\omega_t\leq a-bt\quad \forall\ 0<t<u\rp
\\
&&
\nonumber
\hskip.5cm
=\P^{(\delta)}_0\left(\omega_{1/t}\leq \frac at-b\quad \forall\ \frac 1t>\frac 1u\,\Big
| \omega_{1/u}=0\rp
\\
&&
\hskip.5cm
=\P^{(\delta)}_0\left(\omega_{s}\leq as+\frac{a-bu}{u}\quad \forall\ s>0\rp.
\end{eqnarray}
For $\Delta_r$ it holds (again) by time inversion
\begin{eqnarray*}
&&
\Delta_r
=\P^{(\delta)}_0\left(\omega_{s}\leq \sqrt{a(a-bu)}\lp s+\frac 1u\rp\quad \forall\ s>0\rp.
\end{eqnarray*}
Using hereby scaling (\ref{bes02}) with $c=a/(a-bu)$ we obtain
\begin{eqnarray*}
&&
\Delta_r
=
\P^{(\delta)}_0\left(\omega_{(a-bu)s/a}\leq (a-bu)s+\frac {a-bu}{u}\quad \forall\ s>0\rp,
\end{eqnarray*}
and substituting $t=(a-bu)s/a$ it is seen that $\Delta_r=\Delta_l,$ thus
completing the proof.
\end{proof}

\begin{remark}
We note that for $b\in\R,$ the result of the theorem may be stated as 
$$
M_b\rr\sqrt{M^2_0+\frac{b^2}{4}}-\frac{b}{2},
$$
where $M_b:=\sup_{t\leq 1}\{\omega_t-bt\}.$
\end{remark}

The next result is a slight extension of Exercise 3.10 in Revuz and Yor \cite{revuzyor01} where
reflecting Brownian motion case is considered. We present here also the proof for the readers' convenience.

\begin{thm}
\label{t31}
For Bessel processes with $\delta>0$ it holds 
\begin{equation}
\label{e310}
\P^{(\delta)}_0\left(\sup_ {t\geq 0}\left\{\omega_t-t\right\}<a\right)
=
\P^{(\delta)}_{0,1,0}\left(\sup_{u\leq 1}\left\{\omega_u\right\}<\sqrt{a}\right).
\end{equation}
In other words, letting $R=\{R_t:t\geq 0\}$ denote a Bessel process started from 0 with
dimension parameter $\delta>0$ and $R^{0,1,0}=\{R^{0,1,0}_u\,:\,0\leq u\leq 1\}$ the corresponding Bessel bridge from 0 to 0 of length 1. Then 
\begin{equation}
\label{e31}
\sup_ {t\geq 0}\left\{R_t-t\right\}\,\rr\,\sup_{u\leq 1}\Big\{\lp R^{0,1,0}_u\rp^2\Big\}.
\end{equation}

\end{thm}

\begin{proof}
Consider for $a>0$
\begin{eqnarray*}
&&
\hskip-1cm
\P_0\lp\sup_ {t\geq 0}\{R_t-t\}<a\rp=\P_0\lp \forall
t,\ R_t<a+t\rp
\\
&&
\hskip2.9cm
=\P_0\lp \forall
t,\ R_{\la^2t}<a+\la^2t\rp
\\
&&
\hskip2.9cm
=\P_0\lp \forall
t,\ \la R_{t}<a+\la^2t\rp
\end{eqnarray*}
by scaling for any $\la >0$. Choosing $\la^2=a$ yields
\begin{eqnarray*}
&&
\hskip-1cm
\P_0\lp\sup_ {t\geq 0}\{R_t-t\}<a\rp
=\P_0\lp \forall
t,\ \sqrt a R_{t}<a+a t\rp
\\
&&
\hskip2.9cm
=
\P_0\lp \forall
t,\ \frac {R_t}{1+t}<\sqrt a\rp,
\end{eqnarray*}
i.e.,
\begin{equation}
\label{e32}
\sup_ {t\geq 0}\{R_t-t\}\,\rr\,\lp\sup_ {t\geq 0}\left\{\frac{R_t}{1+t}\right\}\rp^2.
\end{equation}
We now note that
\begin{eqnarray}
\label{e33}
&&
\nonumber
\hskip-.75cm
\sup_{t\geq 0}\left\{\frac{R_t}{1+t}\right\}=\sup_{u\leq 1}\{(1-u)\,R_{u/(1-u)}\}.
\\
&&
\hskip2cm
\,\rr\,\sup_{u\leq 1}\{R^{0,1,0}_u\},
\end{eqnarray}
where, in the first step, we use the substitution  $u=t/(1+t),$ and, for the second step, we use the well known representation
\begin{equation*}
\{R^{0,1,0}_u\,:\, 0\leq u<1\}\,\rr\,\{(1-u)\,R_{u/(1-u)}\,:\, 0\leq u<1\},
\end{equation*}
which is inherited from the corresponding Brownian bridge representation. Claim (\ref{e31}) follows now from (\ref{e32}) and (\ref{e33}).
\end{proof}

\begin{example}
\label{verw}
For 3-dimensional Bessel bridge the distribution of the supremum is
known. In fact, it was proved by Vervaat \cite{vervaat79} that
this supremum is identical in law with the range of a standard
Brownian bridge, i.e., 
$$
\sup_ {t\leq u}\Big\{ R^{0,u,0}_t\Big\}
\,\rr\, 
\sup_ {t\leq u}B^{0,u,0}_t -\inf_ {t\leq u}B^{0,u,0}_t.
$$
The distribution of the range can be deduced from Feller
\cite{feller51}, and we have (see \cite{borodinsalminen02} p. 174) 
$$
\P\lp\sup_ {t\leq u}B^{0,u,0}_t -\inf_ {t\leq u}B^{0,u,0}_t<y\rp
=1+2\sum_{k=1}^\infty\lp 1-4k^2\frac{y^2}u\rp\,\e^{-2k^2y^2/u}.
$$
Consequently, from Theorem \ref{t31} for $d=3$
$$
\P_0(\sup_{t\geq 0}\{R_t-t\}<a)=1+2\sum_{k=1}^\infty\lp 1-
4k^2a\rp\e^{-2k^2a}.
$$ 
\end{example}

\begin{example}
\label{rbm}
For 1-dimensional Bessel bridge, i.e., reflecting Brownian bridge, 
the distribution of the supremum is
also known. From \cite{borodinsalminen02} p. 333 formula 1.1.8 we obtain
$$
\P\lp\sup_ {t\leq u}|B^{0,u,0}_t|<y\rp
=1+2\sum_{k=1}^\infty\lp \e^{-2(2k)^2y^2/u}-\e^{-2(2k-1)^2y^2/u}\rp
$$
and, hence, from Theorem \ref{t31} for $\delta=1$
\begin{eqnarray*}
&&
\P_0(\sup_{t\geq 0}\{|B_t|-t\}<a)
=1+2\sum_{k=1}^\infty\lp \e^{-2(2k)^2a}-\e^{-2(2k-1)^2a}\rp
\\
&&
\hskip3.8cm
=1+2\sum_{k=1}^\infty(-1)^k \e^{-2k^2a}.
\end{eqnarray*}
This formula may also be obtained from (\ref{m02}) by letting $u\to \infty$ and coincides with 
(\ref{de2}).
\end{example}

{\bf Acknowledgement.} We thank Mikhail Stepanov for some numerical
computations and creating the plots of $\Delta_{a,b},$ see Figure 1 and
2.

\bigskip
\bibliographystyle{plain}
\bibliography{yor1}
\end{document}